\documentclass[11pt,a4paper]{article}
\usepackage{amsmath,amssymb,centernot, amsthm}
\usepackage{fullpage}

\newtheorem{thm}{Theorem}[section]
\newtheorem{lem}[thm]{Lemma}

\newtheorem{cor}[thm]{Corollary}

\newtheorem{re}[thm]{Result}
\newtheorem{question}[thm]{Question}

\newtheorem{example}[thm]{Example}

\newtheorem{remark}[thm]{Remark}
\newenvironment{rmk}{\begin{remark} \em}{\end{remark}}

\newcommand{\Z}{\mathbb{Z}}

\DeclareMathOperator{\BH}{ {\rm BH} }

\begin{document}
\title{New Constructions of Group-Invariant \\ Butson Hadamard Matrices}
\author{Tai Do Duc\\ Division of Mathematical Sciences\\
School of Physical \& Mathematical Sciences\\
Nanyang Technological University\\
Singapore 637371\\
Republic of Singapore}
\date{}

\maketitle

 \begin{abstract}
Let $G$ be a finite group and let $h$ be a positive integer.
A $\BH(G,h)$  matrix is a $G$-invariant $|G|\times |G|$ matrix $H$ whose entries are complex $h$th
roots of unity such that $HH^*=|G|I_{|G|}$, where $H^*$ denotes the complex conjugate
transpose of $H$, and  $I_{|G|}$ denotes the identity matrix of order $|G|$.
In this paper, we give three new constructions of $\BH(G,h)$ matrices. 
The first construction is the first known family of $\BH(G,h)$ matrices
in which $G$ does not need to be abelian. 
The second and the third constructions are two families of $\BH(G,h)$ matrices in which $G$ is a finite local ring. 
\end{abstract}

\bigskip

\section{Introduction}

Let $n$ and $h$ be positive integers.
An $n\times n$-matrix $H$ whose entries are complex $h$th roots of unity is called a \textbf{Butson Hadamard matrix} if $HH^* = nI_n$,
where $H^*$ is the complex conjugate transpose of $H$ and $I_n$ is the identity matrix of order $n$.
We say that $H$ is a {\boldmath$\BH(n,h)$} \textbf{matrix}.
The focus of this paper is Butson Hadamard matrices which are invariant under the action of finite groups.
Let $G$ be a finite group. A $|G|\times |G|$ matrix $A=(a_{g,k})_{g,k\in G}$ is
{\boldmath $G$}\textbf{-invariant} if $a_{gl,kl}=a_{g,k}$ for all $g,k,l\in G$.
A $G$-invariant $\BH(|G|,h)$ matrix is called a {\boldmath$\BH(G,h)$} \textbf{matrix}.
The topic of group-invariant Butson Hadamard matrices encompasses many combinatorial objects such as 
generalized Hadamard matrices, generalized Bent functions, abelian splitting semi-regular relative difference sets,
cyclic $n$-roots, perfect sequences and perfect arrays, see \cite{schmidt1} and \cite{duc2} for more details.
\medskip

Let $\exp(G)$ denote the least common multiple of the orders of the elements of $G$.
Let $\nu_p(x)$ denote the $p$-adic valuation of the integer $x$.
Using bilinear forms over finite abelian groups, we \cite[Theorem 2.1]{duc1} constructed $\text{BH}(G,h)$ matrices in which $G$ is a finite abelian group and
\begin{itemize}
\item[(i)] $\nu_p(h) \geq \lceil \nu_p(\exp(G)/2\rceil$ for every prime divisor $p$ of $|G|$, and
\item[(ii)] $\nu_2(h) \geq 2$ if $\nu_2(|G|)$ is odd and $G$ has a direct factor $\mathbb{Z}_2$. 
\end{itemize}
The conditions (i) and (ii) are also necessary conditions for the existence of $\text{BH}(G,h)$ matrices in the case $G$ is a cyclic group of prime-power order, see \cite[Theorem 3.6]{duc1}. Furthermore, \cite[Theorem 2.1]{duc1} is the state-of-the-art class of group-invariant Butson Hamard matrices which includes previously known constructions in \cite{bac} and \cite{lau}.

\medskip

In the case $G$ is a cyclic group, most known necessary conditions, see \cite{duc2}, on the existence of $\text{BH}(G,h)$ matrices involve a condition on $\gcd(|G|,h)$. However when $G$ is not a cyclic group, the situation turns out to be different. As the last two constructions in this paper suggest, the conditions required for these constructions to work are the existence of various vanishing sums of roots of unity.

\medskip

In this paper, we construct three new classes of $\BH(G,h)$ matrices. Our first construction is a class of $\BH(G,h)$ matrices which does not require $G$ to be abelian, but requires $G$ to contain a \textit{large enough} normal cyclic subgroup. The second and the third constructions use ideas in the studies of relative difference sets in finite local rings by Leung et. al. \cite{leu1, leu2}.

\medskip

To end this section, we state the following result by Lam and Leung \cite{lam} on the existence of vanishing sums of roots of unity.

\begin{re}  \label{Lam-Leung} 
Let $n$ and $h$ be positive integers and let $h=\prod_{i=1}^r p_i^{e_i}$ be the prime factorization of $h$. If $\eta_1,\dots, \eta_n$ are $h$th roots of unity such that $\sum_{i=1}^n\eta_i=0$, then $n\in p_1\mathbb{N}+\cdots+p_r\mathbb{N}$.
\end{re}

\section{Preliminaries}

\subsection{Group-invariant Butson Hadamard matrices and group-ring equations}
As it turns out, the properties of group-invariant Butson Hadamard matrices can be wrapped in a single group-ring equation. The theory of group rings and characters is pivotal in our study. 

\medskip

Let $G$ be a finite group, let $R$ be a number ring and let
$R[G]$ denote the group ring of $G$ over $R$. The elements of $R[G]$ have
the form $X=\sum_{g \in G} a_gg$ with $a_g \in R$.
The numbers $a_g$ are called \textbf{coefficients} of $X$.
Two elements $X=\sum_{g \in G} a_gg$ and $Y=\sum_{g\in G}b_gg$ in $R[G]$
are equal if and only if $a_g=b_g$ for all $g\in G$.
A subset $S$ of $G$ is identified with the group ring element $\sum_{g\in S}g$.
For the identity element $1_G$ of $G$ and $\lambda\in R$,  we write $\lambda$
for the group ring element $\lambda 1_G$.

\medskip
For $h\in \Z^+$, we denote a primitive $h$th root of unity by $\zeta_h$. In our study, we focus on the ring $R=\Z[\zeta_h]$.
For $X=\sum_{g \in G}a_gg\in R[G]$, we write
$$
X^{(-1)}=\sum_{g\in G} \overline{a_g} g^{-1},
$$
where $\overline{a_g}$ denotes the complex conjugate of $a_g$.

\medskip

If $G$ is an abelian group, then we can study the ring $R[G]$ using characters of $G$. 
We denote the group of complex characters of $G$ by $\hat{G}$.
The \textbf{trivial character} of $G$, denoted by $\chi_0$, is defined by $\chi_0(g)=1$ for all $g\in G$.
For $D=\sum_{g\in G} a_gg \in R[G]$ and $\chi \in \hat{G}$, we set $\chi(D)=\sum_{g\in G} a_g\chi(g)$.
The following result is the well known Fourier inversion formula for group ring theory and a proof can be found in \cite[Ch.\ VI, Lem. 3.5]{bethjung}.

\begin{re}[Fourier inversion] \label{fourier}
Let $G$ be a finite abelian group and let $D=\sum_{g\in G} a_gg \in \mathbb{C}[G]$.
Then $$a_g=\frac{1}{|G|}\sum_{\chi \in \hat{G}} \chi(Dg^{-1})$$
for all $g\in G$.
Consequently, if $D, E \in \mathbb{C}[G]$ and $\chi(D)=\chi(E)$ for all $\chi \in \hat{G}$, then $D=E$.
\end{re}

The following result \cite[Lemma 3.3]{duc1} translates the properties of a $\BH(G,h)$ matrix into an equation in $\mathbb{Z}[\zeta_h][G]$. We include a proof for the convenience of the reader.

\begin{re} \label{gr}
Let $G$ be a finite group, let $h$ be a positive integer,
and let $a_g$, $g\in G$, be complex $h$th roots of unity. Consider the element
$D = \sum_{g \in G}a_g g$ of $\Z[\zeta_h][G]$
and the matrix $H=(H_{g,k})$, $g,k\in G$, given by $H_{g,k}=a_{gk^{-1}}$.
Then $H$ is a $\BH(G,h)$ matrix if and only if
\begin{equation} \label{gr1}
DD^{(-1)}=|G|.
\end{equation}
\end{re}

\begin{proof}
It is clear that the matrix $H$ defined by $H_{g,k}=a_{gk^{-1}}$ is $G$-invariant, as $H_{gl,kl}=a_{gk^{-1}}=H_{g,k}$ for any $g,k,l\in G$. Let $g\in G$ be arbitrary. The coefficient of $g$ in $DD^{(-1)}$ is
\begin{equation*}
\sum_{k,l \in G \atop kl^{-1}=g}a_{k}\overline{a_l}
= \sum_{l\in G} a_{gl}\overline{a_l}.
\end{equation*}
On the other hand, let $x$ and $y$ be any two elements in $G$. Write $y=gx$, $g\in G$. The inner product of row $y$ and row $x$ of $H$ is
\begin{equation*}
\sum_{k\in G} H_{gx,k}\overline{H_{x,k}}
= \sum_{k\in G} a_{gxk^{-1}}\overline{a_{xk^{-1}}}
 = \sum_{l\in G} a_{gl}\overline{a_l}.
\end{equation*}
Therefore, the equation (\ref{gr1}) holds if and only if any two distinct rows of $H$ have inner product $0$,
that is, if and only if $H$ is a $\BH(G,h)$ matrix.
\end{proof}

\subsection{Our construction approach}
By Result \ref{gr}, the existence of a $\BH(G,h)$ matrix is equivalent to the existence of a group ring element $D=\sum_{g\in G} a_gg$ with the following properties.
\begin{itemize}
\item[(i)] Each $a_g$ is an $h$th root of unity.
\item[(ii)] $DD^{(-1)}=|G|$.
\end{itemize}
In our constructions, we imply a structure on $D$ by putting $D=\sum_{i=1}^t \eta_iG_i$, where
\begin{itemize}
\item[(a)] $G=\cup_{i=1}^t G_i$ (note that we usually choose $G_i$ as subgroups, or cosets of a subgroup, of $G$);
\item[(b)] Each $\eta_i$ is a complex $h$th root of unity.
\end{itemize}
Assume that we have $D=\sum_{i=1}^t \eta_iG_i$ with $\eta_i$ and $G_i$ satisfying conditions (a) and (b). We proceed to finding conditions for $\eta_i$ and $G_i$ such that (i) and (ii) are satisfied. Note that each $g\in G$ belongs to $\cap_{i\in I_g}G_i$ for some nonempty subset $I_g$ of $\{1,\dots,t\}$, so we need $\sum_{i\in I_g} \eta_i=\eta_g$, where $\eta_g$ is an $h$th root of unity. This results in various vanishing sums of roots of unity whose existence is determined by Result \ref{Lam-Leung}. The difficulty now lies in finding $G_i$ such that $DD^{(-1)}=|G|$. An example where this happens is the folowing case: $G$ is a finite abelian group such that for any character $\chi$ of $G$, we have
$$|\chi(G_i)|=\begin{cases} \ 0 \ \text{for any} \ i \in \{1,\dots,t\}\setminus\{j\}, \\ \sqrt{|G|} \ \text{if} \ i=j.\end{cases}$$
Assuming this case, we obtain $\chi(D)=\eta_j\chi(G_j)$ for any $\chi\in \hat{G}$. So $\chi(DD^{(-1)})=|\chi(G_j)|^2=|G|$ for any $\chi\in \hat{G}$, which implies $DD^{(-1)}=|G|$ by Result \ref{fourier}.

\medskip

The idea above will be used in our first and second constructions. The third construction utilizes a similar idea with some modification.

\bigskip

\section{{\boldmath $\BH(G,h)$} Matrices with {\boldmath $G$} non-abelian}
The main idea in the construction of this section comes from modifying the \textit{building set} idea used in the study of relative difference sets, 
see \cite{dav} for more details about building sets.

\medskip

For positive integers $t$ and $h$, we denote the cyclic group of order $t$ by $\Z_t$ and a primitive $h$th root of unity by $\zeta_h$. 
The following lemma is the main ingredient for our construction.

\begin{lem} \label{building blocks}
Let $n$ be a positive integer and let $h$ be the smallest positive integer such that
\begin{equation} \label{relation n and h}
h^2\equiv 0\pmod{n} \ \text{and} \ (\nu_2(n),\nu_2(h))\neq (1,1).
\end{equation}  
Define
\begin{equation} \label{size}
k=\begin{cases} n/h \ \text{if} \ \nu_2(n)\neq 1, \\ 2n/h \ \text{otherwise}. \end{cases}
\end{equation}
Then there exists a set of $k$ group-ring elements $D_i \in \mathbb{Z}[\zeta_h][\mathbb{Z}_{n/k}]$, $0\leq i\leq k-1$, which has the following properties.
\begin{itemize}
\item[(a)] $D_i=\sum_{g\in \mathbb{Z}_{n/k}}a_{ig}g$, where each $a_{ig}$ is a complex $h^{th}$ root of unity.
\item[(b)] $D_iD_j^{(-1)}=0$ for any $0\leq i\neq j\leq k-1$ and $\sum_{i=0}^{k-1}D_iD_i^{(-1)}=n$.
\end{itemize}
\end{lem}

To prove Lemma \ref{building blocks}, we need the following result on square norms of algebraic integers. The first part of this result is a familiar property of quadratic Gauss sums. We include proofs for both parts for the convenience of the reader.

\begin{lem} \label{gauss sum type}
Let $n$ be an odd positive integer and let $b$ be an integer. 
\begin{itemize}
\item[(a)] If $X\in \mathbb{Z}[\zeta_n]$ has form $X=\sum_{i=0}^{n-1} \zeta_n^{i^2+bi}$, then
$$X\bar{X}=n.$$
\item[(b)] If $X \in \mathbb{Z}[\zeta_{4n}]$ has form $X=\sum_{i=0}^{2n-1} \zeta_{4n}^{i^2+2bi}$, then
$$X\bar{X}=2n.$$
\end{itemize} 
\end{lem}
\begin{proof}
(a) We have 
$$X\bar{X} = \sum_{i=0}^{n-1} \sum_{j=0}^{n-1}\zeta_n^{(i+j)^2+b(i+j)-i^2-bi}= \sum_{j=0}^{n-1}\sum_{i=0}^{n-1}\zeta_n^{j^2+bj+2ji}$$
As $n$ is odd, we have $\sum_{i=0}^{n-1}\zeta_n^{2ji}=0$ for any $j\neq 0$. Hence $X\bar{X}=n$.

\medskip

(b) For $0 \leq i \leq 2n-1$, write $i=2j+k$ with $k\in\{0,1\}$ and $j\in\{0,\dots, n-1\}$. We obtain
\begin{eqnarray*}
X &=& \sum_{k=0}^1 \sum_{j=0}^{n-1} \zeta_{4n}^{4j^2+4jk+k^2+4bj+2bk} \\
 &=& \sum_{j=0}^{n-1}\zeta_n^{j^2+bj}+\zeta_{4n}^{2b+1}\sum_{j=0}^{n-1}\zeta_n^{j^2+(b+1)j} 
\end{eqnarray*}
Put $Y=\sum_{j=0}^{n-1}\zeta_n^{j^2+bj}$ and $Z=\sum_{j=0}^{n-1}\zeta_n^{j^2+(b+1)j}$. Then
\begin{equation} \label{decomposition of X}
X= Y+ \zeta_{4n}^{2b+1}Z
\end{equation} 
By part (a), we have
\begin{equation} \label{property of Y,Z}
Y\bar{Y}=Z\bar{Z}=n.
\end{equation} 
Moreover, we have
\begin{eqnarray*}
Z\bar{Y}&=&\sum_{i=0}^{n-1}\sum_{j=0}^{n-1}\zeta_n^{(i+j)^2+(b+1)(i+j)-i^2-bi} \\
&=&\sum_{j=0}^{n-1}\zeta_n^{j^2+(b+1)j}\sum_{i=0}^{n-1}\zeta_n^{(2j+1)i} \\
&=& n\zeta_{4n}^{(n^2-1)+2b(n-1)}.
\end{eqnarray*}
Consequently, we have $Y\bar{Z}=n\zeta_{4n}^{(1-n^2)+2b(1-n)}$. Combining with (\ref{decomposition of X}) and (\ref{property of Y,Z}), we obtain
\begin{eqnarray*} 
X\bar{X}&=&Y\bar{Y}+Z\bar{Z}+\zeta_{4n}^{-2b-1}Y\bar{Z}+\zeta_{4n}^{2b+1}Z\bar{Y} \\
&=& 2n+n(\zeta_{4n}^{-n^2-2bn}+\zeta_{4n}^{n^2+2bn}) \\
&=& 2n.
\end{eqnarray*}
\end{proof}

\medskip

\begin{proof}[Proof of Lemma \ref{building blocks}]
By Result \ref{fourier}, the property (b) of Lemma \ref{building blocks} is equivalent to saying that for any character $\chi$ of $\mathbb{Z}_{n/k}$, we have $\chi(D_i)\neq 0$ for exactly one $i\in\{0,\dots,k-1\}$ and $\chi(D_iD_i^{(-1)})=n$ for this value of $i$.

\medskip

Let $g$ be a generator of $\mathbb{Z}_{n/k}$.
Note that the group of characters of $\Z_{n/k}$ is $\{\chi_t: t=0,\dots,n/k-1\}$, where each $\chi_t: \Z_{n/k}\rightarrow \mathbb{C}$ is defined by $\chi_t(g^i)=\zeta_{n/k}^{ti}$ for any $i=0,\dots,n/k-1$.
Let $m$ be a fixed positive integer which is coprime to $n$.
There are three cases concerning the parity of $\nu_2(n)$ and whether $\nu_2(n) =1$. \\
\textbf{Case 1.} $\nu_2(n)$ is even.

\noindent In this case, we have $h=n/k$ and $\mathbb{Z}_{n/k}=\mathbb{Z}_h$. For each $0\leq i\leq k-1$, define
$$D_i=\sum_{j=0}^{h-1}\zeta_h^{j^2k+mij}g^j\in \mathbb{Z}[\zeta_h][\mathbb{Z}_h].$$
As $h$ is the smallest positive integer with the property (\ref{relation n and h}), we have $\nu_2(h)=\nu_2(n)/2$. Hence, the number $l=h/k=h^2/n$ is odd. Let $\chi$ be any character of $\mathbb{Z}_h$. Assume $\chi(g)=\zeta_h^t$ for some $t \in \mathbb{Z}$. We have $$\chi(D_i)=\sum_{j=0}^{h-1}\zeta_h^{j^2k+j(mi+t)}.$$
Note that $h=lk$. Write $j=lx+y$ with $0\leq x \leq k-1$ and $0\leq y \leq l-1$. We obtain
$$\chi(D_i)=\sum_{y=0}^{l-1}\sum_{x=0}^{k-1}\zeta_h^{y^2k+y(mi+t)+lx(mi+t)}=\sum_{y=0}^{l-1} \zeta_h^{y^2k+y(mi+t)}\sum_{x=0}^{k-1}\zeta_k^{x(mi+t)}.$$
Hence
$$\chi(D_i)=\begin{cases} 0 \ \text{if} \ mi+t \not\equiv 0 \pmod{k},\\ 
 k\sum_{y=0}^{l-1} \zeta_l^{y^2+y(mi+t)/k} \ \text{if} \ mi+t\equiv 0 \pmod{k}. \end{cases}$$
 As $(m,k)=1$, there is only one $i\in \{0,\dots,k-1\}$ with $mi+t\equiv 0 \pmod{k}$. Thus, there is only one $D_i\in \{D_0,\dots,D_{k-1}\}$ such that $\chi(D_i)\neq 0$. Moreover by Lemma \ref{gauss sum type}.(a), the sum $\sum_{y=0}^{l-1} \zeta_l^{y^2+y(mi+t)/k}$ has square norm equal to $l$. We obtain
 $$\chi(D_iD_i^{(-1)})=|\chi(D_i)|^2=k^2l=n.$$

\medskip

\noindent \textbf{Case 2.} $\nu_2(n)$ is odd and $\nu_2(n) \geq 3$. 

\noindent In this case, we have $h=n/k$ and $\mathbb{Z}_{n/k}=\mathbb{Z}_h$. For each $0\leq i\leq k-1$, define
$$D_i=\sum_{j=0}^{h-1}\zeta_h^{j^2k/2+mij}g^j\in \mathbb{Z}[\zeta_h][\mathbb{Z}_h].$$
As $h$ is the smallest positive integer with the property (\ref{relation n and h}), we have $\nu_2(h)=(\nu_2(n)+1)/2$. Note that the group-ring elements $D_i$ are well defined, as $k=n/h\equiv 0\pmod{2}$. Moreover, the number $l=h/(2k)=h^2/(2n)$ is odd. Let $\chi$ be any character of $\mathbb{Z}_h$. Assume $\chi(g)=\zeta_h^t$ for some $t\in \mathbb{Z}$. We have 
$$\chi(D_i)=\sum_{j=0}^{h-1}\zeta_h^{j^2k/2+j(mi+t)}.$$
Note that $h=2lk$. Write $j=2lx+y$ with $0\leq x \leq k-1$ and $0\leq y \leq 2l-1$. We obtain
$$\chi(D_i)=\sum_{y=0}^{2l-1}\sum_{x=0}^{k-1}\zeta_h^{y^2k/2+y(mi+t)+2lx(mi+t)}
=\sum_{y=0}^{2l-1} \zeta_h^{y^2k/2+y(mi+t)}\sum_{x=0}^{k-1}\zeta_k^{x(mi+t)}.$$
Hence
$$\chi(D_i)=\begin{cases} 0 \ \text{if} \ mi+t \not\equiv 0 \pmod{k},\\ 
 k\sum_{y=0}^{2l-1} \zeta_{4l}^{y^2+2y(mi+t)/k} \ \text{if} \ mi+t\equiv 0 \pmod{k}. \end{cases}$$
 There is only one $i\in \{0,\dots,k-1\}$ with $mi+t\equiv 0 \pmod{k}$, so there is only one $D_i\in \{D_0,\dots,D_{k-1}\}$ such that $\chi(D_i)\neq 0$. Moreover by Lemma \ref{gauss sum type}.(b), the sum $\sum_{y=0}^{2l-1} \zeta_{4l}^{y^2+2y(mi+t)/k}$ has square norm equal to $2l$. We obtain
 $$\chi(D_iD_i^{(-1)})=|\chi(D_i)|^2=k^2(2l)=n.$$
 
 \medskip
 
\noindent \textbf{Case 3.} $\nu_2(n)=1$. 

\noindent In this case, we have $h=2n/k$ and $\mathbb{Z}_{n/k}=\mathbb{Z}_{h/2}$. For each $0\leq i\leq k-1$, define
$$D_i=\sum_{j=0}^{h/2-1}\zeta_h^{j^2k+2mij}g^j\in \mathbb{Z}[\zeta_h][\mathbb{Z}_{h/2}].$$
By (\ref{relation n and h}), we have $\nu_2(h)=2$. Hence, the number $l=h/(4k)=h^2/(8n)$ is odd. 
Let $\chi$ be any character of $\mathbb{Z}_{h/2}$. Assume $\chi(g)=\zeta_{h/2}^t=\zeta_h^{2t}$ for some $t\in \mathbb{Z}$. 
We have 
$$\chi(D_i)=\sum_{j=0}^{h/2-1}\zeta_h^{j^2k+2j(mi+t)}.$$
Note that $h=4lk$. Write $j=2lx+y$ with $0\leq x \leq k-1$ and $0\leq y \leq 2l-1$. We obtain
$$\chi(D_i)=\sum_{y=0}^{2l-1}\sum_{x=0}^{k-1}\zeta_h^{y^2k+2y(mi+t)+4lx(mi+t)}
=\sum_{y=0}^{2l-1} \zeta_h^{y^2k+2y(mi+t)}\sum_{x=0}^{k-1}\zeta_k^{x(mi+t)}.$$
Hence
$$\chi(D_i)=\begin{cases} 0 \ \text{if} \ mi+t \not\equiv 0 \pmod{k},\\ 
 k\sum_{y=0}^{2l-1} \zeta_{4l}^{y^2+2y(mi+t)/k} \ \text{if} \ mi+t\equiv 0 \pmod{k}. \end{cases}$$
 There is only one $i\in \{0,\dots,k-1\}$ which satisfies $mi+t \equiv 0 \pmod{k}$, so there is only one $D_i\in \{D_0,\dots,D_{k-1}\}$ such that $\chi(D_i)\neq 0$. Moreover by Lemma \ref{gauss sum type}.(b), the sum $\sum_{y=0}^{2l-1} \zeta_{4l}^{y^2+2y(mi+t)/k}$ has square norm equal to $2l$. We obtain
 $$\chi(D_iD_i^{(-1)})=|\chi(D_i)|^2=k^2(2l)=n.$$

\end{proof}

\begin{thm} \label{nonabel}
Let $n$ and $h$  positive integers with the property 
\begin{equation} \label{padic relation}
h^2\equiv 0\pmod{n} \ \text{and} \ (\nu_2(h),\nu_2(n)) \neq (1,1).  \tag{$\ast$}
\end{equation}
Define the positive divisor $k$ of $n$ as follows.
\begin{equation} \label{size of blocks}
k=\begin{cases} n/h \ \text{if} \ \nu_2(n)\neq 1, \\ 2n/h \ \text{otherwise}.\end{cases}
\end{equation}
Suppose that $G$ is a group of order $n$ which contains $\mathbb{Z}_{n/k}$ as a normal subgroup. Then a $\BH(G,h)$ matrix exists.
\end{thm}

\begin{proof}
Note that a $\BH(G,h)$ matrix is automatically a $\BH(G,h')$ matrix whenever $h\mid h'$, 
as a $h$th root of unity is a power of a $h'$ root of unity. 
Thus, we can assume that $h$ is the smallest positive integer with the property (\ref{padic relation}). 
Let $D_0,\dots,D_{k-1}$ be defined as in Lemma \ref{building blocks}.
Let $x_0,\dots,x_{k-1}$ be a full set of coset representatives of $\mathbb{Z}_{n/k}$ in $G$. Define
$$D=\sum_{i=0}^{k-1}D_ix_i.$$
Note that $D_i=\sum_{x\in \Z_{n/k}} a_{ix}x$. So $D=\sum_{i=0}^{k-1}\sum_{x\in \Z_{n/k}}a_{ix}xx_i$. Any element $h\in G$ has a unique expression $h=xx_i$ with $x\in \Z_{n/k}$ and $i\in \{0,\dots,k-1\}$. Hence the coefficient of $h$ in $D$ is $a_{ix}$, an $h$th root of unity by Lemma \ref{building blocks}.(a). By Result \ref{gr}, to show that $D$ gives a $\BH(G,h)$ matrix, it remains to verify that $DD^{(-1)}=n$.
We have 
$$DD^{(-1)}=\sum_{0\leq i,j\leq k-1} D_ix_ix_j^{-1}D_j^{(-1)}=\sum_{0\leq i,j\leq k-1} x_ix_j^{-1}D_iD_j^{(-1)},$$
where the second equality follows from the fact that $D_i\in \mathbb{Z}[\zeta_h][\mathbb{Z}_{n/k}]$ and $\mathbb{Z}_{n/k}$ is a normal subgroup of $G$.
By Lemma \ref{building blocks}.(b), we have $D_iD_j^{(-1)}=0$ whenever $i\neq j$. Hence
$$DD^{(-1)}=\sum_{i=0}^{k-1}D_iD_i^{(-1)}=n.$$
\end{proof}

\bigskip

\section{Construction in Finite Local Rings}

In this section, we adopt the notations and properties of finite local rings used by Leung and Ma \cite{leu1}. A ring $R$ is called a \textbf{finite local ring} if it contains a finite number of elements and it has a unique maximal ideal. Our goal in this section is to construct two classes of $\BH(R\times R,h)$ matrices, where $h$ is some positive integer. The following result \cite[Propositions 2.3, 2.4]{leu1} guarantees the existence of a finite local ring and describes its basic properties.

\begin{re} \label{existence}
Let $d,m,n$ be positive integers. Then for any prime $p$, there exists a finite local ring $R$ of order $p^{dn}$ with the following properties.
\begin{itemize}
\item[(i)] Its maximal ideal $I$ is generated by a prime element $\pi$;
\item[(ii)] $\pi^{n-1}\neq 0, \pi^n=0$ and $p=\pi^mu$, where $u$ is a unit in $R$;
\item[(iii)] $R/I$ is a finite field of order $p^d$;
\item[(iv)] Write $n=am+b$, $0\leq b\leq m-1$, $a\in \mathbb{N}$. The group $(R,+)$ is isomorphic to $\left(\Z_{p^{a+1}}\right)^{db}\times \left(\Z_{p^a}\right)^{d(m-b)}$.
\end{itemize}
\end{re}
From now on, we fix the local ring $R$ with parameters $I, \pi, d,m,n,a,b$ as above. We note that any nontrivial ideal of $R$ has the form $I^{t}$ for some $1\leq t\leq n-1$. Moreover $|I^t|=p^{d(n-t)}$ for any $1\leq t\leq n-1$. Next, we look at character groups of $R$ and $R\times R$.

\medskip

Let $\tau$ be an additive character of $R$ which is nontrivial on $I^{n-1}$. For each $a\in R$, define the map $\tau_a: R \rightarrow \mathbb{C}$ by $\tau_a(x)=\tau(ax)$ for all $x\in R$. It is straightforward to check that $\tau_a$ is an additive character of $R$.
Moreover for any $a,b\in R$, we have $\tau_a=\tau_b$ if and only if $\{(a-b)x:x\in R\}\subset \text{ker}(\tau)$.
As $\{(a-b)x: x\in R\}$ is an ideal of $R$ and $\tau$ is nontrivial on $I^{n-1}$, we have $\{(a-b)x: x\in R\}=\{0\}$, which implies $a=b$. Therefore, the group of characters of $R$ is $\{\tau_a: a\in R\}$.

\medskip

Let $\chi$ be any character of $R\times R$. As $\chi$ is a character on $R\times \{0\}$, there exists a unique $a\in R$ such that $\chi(x,0)=\tau_a(x)$ for any $x\in R$. Similarly, there exists a unique $b\in R$ such that $\chi(0,y)=\tau_b(y)$ for any $y\in R$.
For any $(x,y)\in R\times R$, we have
$$\chi(x,y)=\chi(x,0)\chi(0,y)=\tau_a(x)\tau_b(y)=\tau(ax+by).$$
Therefore, each character $\chi$ of $R\times R$ has the form $\tau_{a,b}$, $a,b\in R$, with $\tau_{a,b}$ being defined by $\tau_{a,b}(x,y)=\tau(ax+by)$ for any $(x,y)\in R\times R$.
As the character group of $R\times R$ has size $|R\times R|$, this group is $\{\tau_{a,b}:a,b\in R\}$.

\subsection{Construction 1}
Note that each element $x\in R\setminus\{0\}$ can be written uniquely as $x=\pi^k u$ for some positive integer $k\leq n-1$ and a unit $u$ in $R$. 
The following result \cite[Section 4]{leu1} partitions $R\times R$ in a way which becomes the main ingredient for our construction in this subsection.

\begin{re} \label{partition}
Let $\phi: R\rightarrow R$ be defined by $x=\pi^k u \mapsto \phi(x)=\pi^ku^{-1}$. 
Let $t\leq d$ be a positive integer and let $\{R_0,\dots, R_{p^t-1}\}$ be a partition of $R$ such that for any coset $a+I^{n-1}$ of $I^{n-1}$ in $R$ and any $i\in \{0,\dots, p^t-1\}$, we have $|R_i\cap (a+I^{n-1})|=p^{d-t}$. By $f_i$ we denote the characteristic function of $R_i$, that is, $f_i(x)=1$ if $x\in R_i$ and $f_i(x)=0$ otherwise.
For each $i=0,\dots,p^t-1$, define
\begin{equation} \label{compo}
D_i=\{(x,y)\in R\times R: f_i(\phi(x)y)=1\}.
\end{equation}
Then $\{D_0,\dots,D_{p^t-1}\}$ is a partition of $R\times R$ and for any character $\chi$ of $R\times R$, we have
\begin{equation} \label{charvalue}
\chi(D_i)=\begin{cases} p^{dn}f_i(0)+p^{dn-t}(p^{dn}-1) \ \text{if} \ \chi=\chi_0, \\ p^{dn}f_i(c_\chi)-p^{dn-t} \ \text{if} \ \chi\neq \chi_0,\end{cases}
\end{equation}
where the element $c_\chi \in R$ depends only on $\chi$.
\end{re}

\begin{thm} \label{local1}
Let $R$ be a finite local ring defined as in Result \ref{existence}.
Let $h$ be a positive integer and let $h=\prod_{i=1}^r p_i^{e_i}$ be the prime factorization of $h$.
Assume that 
\begin{equation} \label{nec1}
p^d\in p_1\mathbb{N}+\cdots+p_r\mathbb{N}.
\end{equation}
Then a $\BH(R\times R,h)$ matrix exists.
\end{thm}

\begin{proof}
By (\ref{nec1}) and Result \ref{Lam-Leung}, there exists a vanishing sum of roots of unity of length $p^d$. In fact, these $p^d$ roots of unity can be used directly in our construction for $\BH(R\times R,h)$ matrices. However, we will prove this theorem under a more flexible condition as follows. Assume that $1\leq t\leq d$ is any integer such that $p^t\in p_1\mathbb{N}+\cdots+ p_r\mathbb{N}$. By Result \ref{Lam-Leung}, there exists $p^t$ $h$th roots of unity $\eta_0, \dots,\eta_{p^t-1}$ such that 
\begin{equation} \label{nec local ring}
\sum_{i=0}^{p^t-1}\eta_i=0.
\end{equation}
We show that these $\eta_i$ can be used to construct a $\BH(R\times R,h)$ matrix. Let $D_i$, $0\leq i\leq p^t-1$, be subsets of $R\times R$ defined as in (\ref{compo}). Define
$$D=\sum_{i=0}^{p^t-1}\eta_iD_i \in \mathbb{Z}[\zeta_h][R\times R].$$
By Result \ref{partition}, the set $\{D_0,\dots,D_{p^t-1}\}$ is a partition of $R\times R$. Let $g\in R\times R$ and assume $g\in D_i$ for some $i\in \{0,\dots,p^t-1\}$.
As $g$ does not belong to any $D_j$ with $j\neq i$, the coefficient of $g$ in $D$ is $\eta_i$, an $h$th root of unity.
Thus all coefficients in $D$ are complex $h$th roots of unity.
To show that $D$ gives a $\BH(R\times R,h)$ matrix, it remains, by Result \ref{gr}, to show that $DD^{(-1)}=|R\times R|$.
\medskip

Let $\chi_0$ denote the trivial character of $R\times R$. By (\ref{charvalue}), we have 
$$\chi_0(D)=p^{dn}\left(\sum_{i=0}^{p^t-1}\eta_i f_i(0)\right)+p^{dn-t}(p^{dn}-1)\left(\sum_{i=0}^{p^t-1}\eta_i\right)=p^{dn}\left(\sum_{i=0}^{p^t-1}\eta_i f_i(0)\right),$$ 
where the last equality follows from (\ref{nec local ring}). Note that $\{R_0,\dots,R_{p^t-1}\}$ is a partition of $R$ and $f_i$ is the characteristic function of $R_i$. So there exists $j\in \{0,\dots,p^t-1\}$ such that $f_j(0)=1$ and $f_i(0)=0$ for any $i\neq j$.
We obtain $\chi_0(D)=p^{dn}\eta_j$, which implies $\chi_0(DD^{(-1)})=p^{2dn}$. 

\medskip

Let $\chi$ be any nontrivial character of $R\times R$. By (\ref{charvalue}), we have
$$\chi(D)=p^{dn}\left(\sum_{i=0}^{p^t-1}\eta_if_i(c_\chi)\right)-p^{dn-t}\left(\sum_{i=0}^{p^t-1}\eta_i\right)=p^{dn}\left(\sum_{i=0}^{p^t-1}\eta_if_i(c_\chi)\right).$$ 
Similar to the previous case, there exists $j\in \{0,\dots,p^t-1\}$ such that $f_j(c_\chi)=1$ and $f_i(c_\chi)=0$ for any $i\neq j$. We obtain $\chi(D)=p^{dn}\eta_j$, which implies $\chi(DD^{(-1)})=p^{2dn}$. 

\medskip

Therefore, we have $\chi(DD^{(-1)})=p^{2dn}=|R\times R|$ for any character $\chi$ of $R\times R$. As $R\times R$ is an abelian group, we obtain, by Result \ref{fourier},
$$DD^{(-1)}=|R\times R|.$$

\end{proof}

\medskip

\begin{rmk} \label{exist1}
The construction in Theorem \ref{local1} depends on the relation $p^d\in p_1\mathbb{N}+\cdots+p_r\mathbb{N}$, where $p_1<\cdots<p_r$ are all prime divisors of $h$. In general, there are many cases where this can happen. For example, if $p^d\geq (p_1-1)(p_2-1)$, then we can write $p^d=p_1a_1+p_2a_2$ for some $a_1,a_2\in \mathbb{N}$, see \cite[Ex. 4, p.22]{lev}, which implies $p^d\in p_1\mathbb{N}+\cdots+p_r\mathbb{N}$. Specifically, consider the case $h\equiv 0\pmod{6}$. We have $p_1=2$ and $p_2=3$ and $p^d\geq (p_1-1)(p_2-1)=2$ for any prime $p$ and positive integer $d$. Hence $p^d\in p_1\mathbb{N}+\cdots+p_r\mathbb{N}$ and a $\BH(R\times R,h)$ matrix exists.
\end{rmk}

\bigskip

\subsection{Construction 2}
For each $r\in R$ and $s\in I$, define 
$$I_r=\{(x,xr): x\in R\}, \ J_s=\{(xs,x): x\in R\}.$$
Note that both $I_r$ and $J_s$ are subgroups of $R\times R$ and each has order $|R|=p^{dn}$.
Recall that by Result \ref{gr}, the existence of a $\BH(R\times R,h)$ matrix is equivalent to the existence of a group-ring element $D\in \Z[\zeta_h][R\times R]$ which has coefficients as $h$th roots of unity and satisfies $DD^{(-1)}=|R\times R|$.
In our coming construction, we put $D=\sum_{r\in R}\eta_rI_r+\sum_{s\in I}\mu_sJ_s$, where $\eta_r$ and $\mu_s$ are $h$th roots of unity and satisfy certain conditions. 

\medskip

Recall that $\pi$ is the generator of the maximal ideal $I$ of $R$. For any $x\in R$, we can write it uniquely in the form $x=\pi^tu_x$, where $u_x$ is a unit in $R$ and $t\in \{0,\dots,n\}$. Define $\nu_\pi: R\rightarrow \mathbb{N}$ by $x=\pi^tu_x\mapsto \nu_\pi(x)=t$.
\medskip

First, observe the following properties of the subgroups $I_r$ and $J_s$ of $R\times R$.

\begin{lem}\label{intersect}
The sets $I_r$ and $J_s$ have the following properties.
\begin{itemize}
\item[(a)] For any $(x,y)\in R\times R\setminus\{(0,0)\}$, we have
$$(x,y)\in \begin{cases} \cup_{r\in R}I_r \ \text{if} \ \nu_\pi(x)\leq \nu_\pi(y),\\ \cup_{s\in I}J_s \ \text{if} \ \nu_\pi(x)>\nu_\pi(y). \end{cases}$$
\item[(b)] $I_r\cap J_s=\{(0,0)\}$ for any $r\in R$ and $s\in I$.
\item[(c)] Let $(x,xr)$ be an element in $I_r$ and put $t=\nu_\pi(x)$. Then
$$\{I_{r'}: (x,xr)\in I_{r'}\}=\{I_{r'}: r'\in r+I^{n-t}\}.$$
Let $(xs,x)$ be an element in $J_s$ and put $t=\nu_\pi(x)$. Then
$$\{J_{s'}: (xs,x)\in J_{s'}\}=\{J_{s'}: s'\in s+I^{n-t}\}.$$
\end{itemize}
\end{lem}
\begin{proof}
(a) If $\nu_\pi(x)\leq \nu_\pi(y)$, there exists $r\in R$ such that $y=xr$. Hence $(x,y)=(x,xr)\in I_r$. On the other hand, if $\nu_\pi(x)>\nu_\pi(y)$, there exists $s\in I$ such that $x=ys$. Hence $(x,y)=(ys,y)\in J_s$. 

\medskip

\noindent (b) Assume that $(x,xr)\in I_r\setminus\{(0,0)\}$ is an element in $J_s$ for some $s\in I$. By part (a), we have $\nu_\pi(x)>\nu_\pi(xr)$, a contradiction. Thus $I_r\cap J_s=\{(0,0)\}$. 

\medskip

\noindent (c) Assume that $(x,xr) \in I_r$ is an element in $I_{r'}$. We obtain $x(r'-r)=0$, which implies $r'\in r+I^{n-t}$, where $t=\nu_\pi(x)$. Hence $\{I_{r'}: (x,xr) \in I_{r'}\}=\{I_{r'}: r'\in r+I^{n-t}\}$. The remaining claim is proved in the same way.
\end{proof}

\medskip

Let $R_1\subset R$ be a fixed set of coset representatives of $I$ in $R$ such that $R_1$ contains $0$. Define $R_0=\{0\}$. For $1\leq i \leq n$, define
$$R_i=R_1+\pi R_1+\dots+\pi^{i-1}R_1.$$
Hence $R_i$ is a set of coset representatives of $I^i$ in $R$. Moreover, the set $\pi R_j$, $0\leq j\leq n-1$, is a set of coset representatives of $I^{j+1}$ in $I$. Note that $\{0\}=R_0\subset R_1\subset R_2\subset \cdots\subset R_{n-1}\subset R_n=R$ and $|R_i|=p^{di}$ for any $0\leq i\leq n$. 
We are ready for the main result of this subsection.

\medskip

\begin{thm} \label{local2}
Let $\eta$, $\eta_r$, $\delta_u$, $\mu_s$, $\gamma_v$ with $r\in R$, $u\in R_{n-1}$, $s\in I$, $v\in \pi R_{n-2}$, be complex $h$th roots of unity such that for any $1\leq i\leq n-1$ and $0\leq j\leq n-2$, we have
\begin{eqnarray}
\sum_{r\in u+I^i}\eta_r &=& \delta_u \ \text{for any} \ u\in R_i , \label{first} \\
\sum_{s\in v+I^{j+1}}\mu_s &=& \gamma_v \ \text{for any} \ v\in \pi R_j, \label{second}\\\
\sum_{u\in R_1}\delta_u+\sum_{v\in \pi R_1}\gamma_v &=& \eta. \label{third}
\end{eqnarray}
Put
$$D=\sum_{r\in R}\eta_rI_r+\sum_{s\in I}\mu_sJ_s.$$
Then $D$ has coefficients as complex $h$th roots of unity and 
$$DD^{(-1)}=|R\times R|.$$
As a consequence, there exists a $\BH(R\times R,h)$ matrix.
\end{thm}

\begin{proof}
First, we prove that $D$ has all coefficients as complex $h$th roots of unity. Let $(x,y)$ be any nonzero element in $R\times R$. Comparing the values of $\nu_\pi(x)$ and $\nu_\pi(y)$, we consider the following two cases.

\medskip

\noindent \textbf{Case 1.} $\nu_\pi(x)\leq \nu_\pi(y)$. \\
 Write $t=\nu_\pi(x)$ and $y=xr$ for some $r\in R$. Note that $0\leq t\leq n-1$, as $(x,y)\neq (0,0)$. By (a) and (b) of Lemma \ref{intersect}, we have $(x,y)\in \left(\cup_{r\in R}I_r\right)\setminus \left(\cup_{s\in I}J_s\right)$. By Lemma \ref{intersect}.(c), we have $(x,y)\in I_s$ if and only if $s\in r+I^{n-t}$. If $t=0$, then $I_r$ is the only set which contains $(x,y)$ and the coefficient of $(x,y)$ in $D$ is $\eta_r$. Assume that $1\leq t\leq n-1$. By equation (\ref{first}), the coefficient of $(x,y)$ in $D$ is
 $$\sum_{s\in r+I^{n-t}}\eta_s=\delta_{u},$$
 where $u\in R_{n-t}$ such that $r-u\in I^{n-t}$.
 
 \medskip
 
 \noindent \textbf{Case 2.} $\nu_\pi(x)>\nu_\pi(y)$.\\
Write $t=\nu_\pi(y)$ and $x=ry$ with some $r\in I$. Note that $0\leq t\leq n-1$. By (a) and (b) of Lemma \ref{intersect}, we have $(x,y)\in \left(\cup_{s\in I}J_s\right)\setminus \left(\cup_{r\in R}I_r\right)$.
By Lemma \ref{intersect}.(c), we have $(x,y)\in J_s$ if and only if $s\in r+I^{n-t}$. If $t=0$, then $J_r$ is the unique set which contains $(x,y)$ and the coefficient of $(x,y)$ in $D$ is $\mu_r$. Assume that $1\leq t\leq n-1$. By equation (\ref{second}), the coefficient of $(x,y)$ in $D$ is
$$\sum_{s\in r+I^{n-t}}\mu_s=\gamma_{v},$$
where $v\in \pi R_{n-1-t}$ such that $r-v\in I^{n-t}$.
 \medskip
 
 We have shown that any nonzero element $(x,y)$ has coefficient in $D$ as an $h$th root of unity. Now we consider the coefficient of $(0,0)$ in $D$.

\noindent \textbf{Case 3.} $(x,y)=(0,0)$.\\
As $(0,0)\in I_r$ and $(0,0)\in J_s$ for any $r\in R$, $s\in I$, its coefficient in $D$ is $\sum_{r\in R}\eta_r+ \sum_{s\in I}\mu_s$. By (\ref{first}), we have $\sum_{r\in R}\eta_r=\sum_{u\in R_1}\sum_{r\in u+I}\eta_r=\sum_{u\in R_1}\delta_u$. By (\ref{second}), we have $\sum_{s\in I}\mu_s=\sum_{v\in \pi R_1}\sum_{s\in v+I^2}\mu_s=\sum_{v\in \pi R_1}\gamma_v$. By (\ref{third}), we obtain
 $$\sum_{r\in R}\eta_r+ \sum_{s\in I}\mu_s=\sum_{u\in R_1}\delta_u+\sum_{v\in \pi R_1}\gamma_v=\eta,$$
 proving that the coefficient of $(0,0)$ in $D$ is an $h$th root of unity. 
 
 \medskip
  \medskip
 
It remains to show that $DD^{(-1)}=|R\times R|$, which is equivalent to showing $|\chi(D)|^2=p^{2dn}$ for any character $\chi$ of $R\times R$. Recall that the group of characters of $R\times R$ is $\{\tau_{a,b}: a,b\in R\}$. Each character $\tau_{a,b}$ is defined by $\tau_{a,b}(x,y)=\tau(ax+by)$, where $\tau$ is a fixed character of $R$ which is nontrivial on $I^{n-1}$. The proof of $|\tau_{a,b}(D)|^2=p^{2dn}$ for any $a,b\in R$ follows from the following claims.
 
\medskip

\noindent \textbf{Claim 1.} $|\tau_{0,0}(D)|^2=p^{2dn}$. 

\medskip

\noindent We have $\tau_{0,0}(D)=p^{dn}\left( \sum_{r\in R}\eta_r+\sum_{s\in I}\mu_s \right)=p^{dn}\eta$.
Hence $$|\tau_{0,0}(D)|^2=p^{2dn}.$$

\medskip

From now one, we fix $(a,b)\in R\times R\setminus\{(0,0)\}$. 

\medskip

\noindent \textbf{Claim 2.} 
There exists $r\in R$ or $s\in S$ such that $\tau_{a,b}(I_r)=p^{dn}$ or $\tau_{a,b}(J_s)=p^{dn}$, respectively. Moreover, the two equations $\tau_{a,b}(I_r)=p^{dn}$ and $\tau_{a,b}(J_s)=p^{ds}$ cannot happen simultaneously. 

\medskip

\noindent If $\tau_{a,b}(I_r)=p^{dn}$ for some $r\in R$, then $\tau((a+br)x)=1$ for all $x\in R$, which implies 
\begin{equation} \label{cond1}
a+br=0.
\end{equation}
If $\tau_{a,b}(J_s)=p^{dn}$ for some $s\in I$, then $\tau((as+b)x=1)$ for all $x\in R$, which implies
\begin{equation} \label{cond2}
as+b=0.
\end{equation}
The equation (\ref{cond1}) has a solution $r\in R$ if and only if $\nu_\pi(a)\geq \nu_\pi(b)$. The equation (\ref{cond2}) has a solution $s\in I$ if and only if $\nu_\pi(a)<\nu_\pi(b)$. Thus in any case, only one of the two equations (\ref{cond1}) or (\ref{cond2}) has solution.
\medskip

\noindent \textbf{Claim 3.} $|\tau_{a,b}(D)|^2=p^{2dn}$.

\medskip

\noindent First, we assume that $\nu_\pi(a)\geq \nu_\pi(b)$. Write $t=\nu_\pi(b)$, $0\leq t\leq n-1$. By the proof of Claim 2, there exists an element $r\in R$ such that $\tau_{a,b}(I_r)=p^{dn}$. We also have
\begin{equation} \label{trivial}
\{I_s: \tau_{a,b}(I_s)=p^{dn}\}=\{I_s: a+bs=0\}=\{I_s: s\in r+I^{n-t}\}.
\end{equation}

\noindent For any $s\not\in r+I^{n-t}$, we have $a+bs\neq 0$ and
\begin{equation} \label{nontri1}
\tau_{a,b}\left(I_{s}\right)=\sum_{x\in R}\tau\left((a+bs)x\right)=\sum_{x\in R}\tau_{a+bs}(x)=0,
\end{equation}
where in the last equality, we use the property that $\tau_{a+bs}$ is a nontrivial character of $R$. Similarly, for any $s\in I$, we have $as+b\neq 0$ and
\begin{equation} \label{nontri2}
\tau_{a,b}(J_s)=\sum_{x\in R}\tau\left((as+b)x\right)=\sum_{x\in R}\tau_{as+b}(x)=0.
\end{equation}
Combining (\ref{trivial}), (\ref{nontri1}) and (\ref{nontri2}), we obtain
$\tau_{a,b}(D)=p^{dn}\left(\sum_{s\in r+I^{n-t}}\eta_{s}\right)$. If $t=0$, then $\tau_{a,b}(D)=p^{dn}\eta_r.$ If $1\leq t\leq n-1$, then $\tau_{a,b}(D)=p^{dn}\delta_{u},$
where $u\in R_{n-t}$ such that $r-u\in I^{n-t}$. In any case, we obtain $|\tau_{a,b}(D)|^2=p^{2dn}.$
 
\medskip

Lastly, the case $\nu_\pi(a)<\nu_\pi(b)$ follows similarly from the last case as follows. Put $t=\nu_\pi(a)$, $0\leq t\leq n-1$. Fix $s\in I$ such that $as+b=0$, that is, $\tau_{a,b}(J_s)=p^{dn}$. We have
$$\{J_{s'}: \tau_{a,b}(J_{s'})=p^{dn}\}=\{J_{s'}: s'\in s+I^{n-t}\}.$$ 
For any $s'\in I \setminus (s+I^{n-t})$, we have $\tau_{a,b}(J_{s'})=0$, and for any $r\in R$, we have $\tau_{a,b}(I_r)=0$. We obtain
$$\tau_{a,b}(D)=p^{dn} \left(\sum_{s'\in s+I^{n-t}}\mu_{s'}\right).$$
If $t=0$, then $\tau_{a,b}(D)=p^{dn}\mu_s$. If $1\leq t\leq n-1$, then $\tau_{a,b}(D)=p^{dn}\gamma_v$, where $v\in \pi R_{n-1-t}$ such that $v-s\in I^{n-t}$. In any case, we obtain $|\tau_{a,b}(D)|^2=p^{2dn}$.

\medskip

We finish the proof of Claim 3 and finish the proof of Theorem \ref{local2}.
\end{proof}

\begin{rmk} \label{exist2}
The construction of $\BH(R\times R,h)$ matrices in Theorem \ref{local2} depends on the existence of complex $h$th roots of unity $\eta_r$, $\delta_u$, $\mu_s$, $\gamma_v$ which satisfy equations (\ref{first}), (\ref{second}) and (\ref{third}). In general, there can be various choices for these roots of unity. 

\medskip

For example, assume that $h\equiv 0 \pmod{6}$. Note that (\ref{first}), (\ref{second}) and (\ref{third}) are equations of vanishing sums of $h$th roots of unity of various lengths.
By Result \ref{Lam-Leung}, a vanishing sum of $h$th roots of unity of length $n$ exists if and only if $n \in p_1\mathbb{N}+\cdots+p_r\mathbb{N}$, where $p_i$'s are all prime divisors of $h$.
As any integer $n\geq 2$ can be written in the form $2a+3b$ for some $a,b\in \mathbb{N}$, we have $n \in p_1\mathbb{N}+\cdots+p_r\mathbb{N}$, which implies the existence of a vanishing sum of $h$th roots of unity of length $n$.
 Hence there exist complex $h$th roots of unity $\eta_r$, $\delta_u$, $\mu_s$, $\gamma_v$ which satisfy (\ref{first}), (\ref{second}) and (\ref{third}).
\end{rmk}

\bigskip

\subsection{A new family of perfect arrays}
A multi-dimensional array $A=(a_{i_1,\dots,i_k})$ of size $n_1\times\dots\times n_k$ is called a \textbf{perfect \boldmath$h$-phase array} if all its entries are complex $h$th roots of unity and
$$\sum_{0\leq i_j\leq n_j-1 \ \forall \ j} a_{i_1,\dots,i_k}\overline{a}_{i_1+s_1,\dots,i_k+s_k}=0$$
 whenever $(s_1,\dots,s_k) \neq (0,\dots,0)$, where the indices are taken modulo $n_j$ for $1\leq j \leq k$. Perfect arrays have a wide range of applications in communication and radar systems, see \cite{bomer}, \cite{fan}, \cite{golomb} for example.

\medskip

In \cite[Lemma 3.4]{duc2}, the author shows the following equivalence between perfect arrays and group-invariant Butson Hadamard matrices.

\begin{re} \label{equiv}
Let $k,h, n_1,\ldots ,n_k$ be positive integers. Then a perfect $h$-phase array of size $n_1\times \cdots \times n_k$ exists if and only if a $\BH(\mathbb{Z}_{n_1}\times \cdots \times \mathbb{Z}_{n_k},h)$ matrix exists.
\end{re}
Using Result \ref{equiv} and the construction of group-invariant Butson Hadamard matrices in \cite{duc1}, the author \cite{duc2} obtained the following array. 

\begin{re}
Suppose that $k,h,n_1,\ldots,n_k$ are positive integers such that
\begin{equation*}
(h,n_i)^2\equiv 0\pmod{n_i} \ \text{and} \ (\nu_2(n_i),\nu_2(h)) \neq (1,1)  \ \text{for any} \ 1\leq i \leq k. 
\end{equation*}
Then a perfect $h$-phase array of size $n_1\times \cdots \times n_k$ exists.
\end{re}

\medskip

In Theorem \ref{local1} and Theorem \ref{local2}, we constructed $\BH(R\times R,h)$ matrices in which $R$ is a finite local ring with additive group structure $R\cong \left(\Z_{p^{a+1}}\right)^{db}\times \left(\Z_{p^a}\right)^{d(m-b)}$. Hence the corresponding perfect arrays can be constructed directly using Result \ref{equiv}.

\medskip

The construction of $\BH(R\times R,h)$ matrices in Theorem \ref{local1} and Theorem \ref{local2} depends on either following conditions.

\begin{itemize}
\item[(i)] $p^d\in p_1\mathbb{N}+\cdots+p_r\mathbb{N}$, where $p_1,\dots,p_r$ are all prime divisors of $h$.
\item[(ii)] There exists complex $h$ roots of unity $\eta$, $\eta_r$, $\delta_u$, $\mu_s$, $\mu_v$ with $r\in R$, $u\in R_{n-1}$, $s\in I$, $v\in \pi R_{n-2}$, such that for any $1\leq i\leq n-1$ and $0\leq j\leq n-2$, we have
\begin{eqnarray*}
\sum_{r\in u+I^i}\eta_r &=& \delta_u \ \text{for any} \ u\in R_i ,  \\
\sum_{s\in v+I^{j+1}}\mu_s &=& \gamma_v \ \text{for any} \ v\in \pi R_j, \\
\sum_{u\in R_1}\delta_u+\sum_{v\in \pi R_1}\gamma_v &=& \eta. 
\end{eqnarray*}
\end{itemize}

\medskip

\begin{thm}  \label{newperfectarray}
Under the same notations as before, assume that either condition (i) or (ii) is satisfied. Then there exists a perfect $h$-phase array of size $p^{a+1}\times\cdots\times p^{a+1}\times p^a\times \cdots \times p^a$, where there are $2db$ terms $p^{a+1}$ and $2d(m-b)$ terms $p^a$.
\end{thm}

\medskip

We note that in the case $h\equiv 0 \pmod{6}$, both conditions (i) and (ii) are satisfied, as discussed in Remark \ref{exist1} and Remark \ref{exist2}. We have the following corollary.

\begin{cor}
Let $a$ and $e$ be nonnegative integers. Let $f$ and $h$ be positive integers such that $h\equiv 0 \pmod{6}$. Let $p$ be a prime. Then there exists a perfect $h$-phase array of size $p^{a+1}\times\cdots\times p^{a+1}\times p^a\times \cdots \times p^a$, where there are $2e$ terms $p^{a+1}$ and $2f$ terms $p^a$.
\end{cor}

\medskip

\noindent \textbf{Acknowledgment.} The author is grateful to Bernhard Schmidt for countless discussions and encouragement throughout this project. The second local-ring construction in this paper is inspired by his work on $\BH(\Z_{p^a}\times \Z_{p^a},h)$ matrices.


\begin{thebibliography}{10}
\bibitem{bac} J. Backelin:
Square Multiples $n$ Give Infinite Many Cyclic $n$-roots. \textit{Reports, Matematiska Institutionen, Stockholms Universitet}, \textbf{8} (1989), 1 -- 2.

\bibitem{bethjung} T.\ Beth, D.\ Jungnickel, H.\ Lenz:
\textit{Design Theory} (2nd edition), Cambridge University Press 1999.

\bibitem{bomer} L.\ Bomer, M.\ Antweiler: Perfect N-phase sequences and arrays [spread spectrum communication], \textit{IEEE J. Selected Areas Commun.} \textbf{10} (1992), 782 -- 789.

\bibitem{dav} J. A. Davis, J. Jedwab:
A unifying construction for difference sets.
J. Combin. Theory Ser. A \textbf{80} (1997), 13 -- 78.

\bibitem{duc1} T. Do Duc, B. Schmidt:
Bilinear Forms on Finite Abelian Groups and Group-Invariant Butson Hadamard Matrices. 
\textit{J. Comb. Theory Ser. A}. \textbf{166} (2019), 337 -- 351. 

\bibitem{duc2} T. Do Duc:
Necessary conditions for the existence of group-invariant Butson Hadamard matrices and a new family of perfect arrays. \textit{Des. Codes Cryptogr.} (2019). https://doi.org/10.1007/s10623-019-00671-4

\bibitem{fan} P.\ Z.\ Fan, M.\ Darnell: \textit{Sequence Design for Communications Applications}. Hoboken, NJ, USA: Wiley, 1996.

\bibitem{golomb} S.\ W.\ Golomb, G.\ Gong: \textit{Signal Design for Good Correlation: For Wireless Communication, Cryptography and Radar}. New York, NY, USA: Cambridge Univ. Press, 2005.

\bibitem{lam} T. Y. Lam, K. H. Leung: 
On vanishing sums of roots of unity
\textit{J. Algebra} \textbf{224} (2000), 91--109.

\bibitem{lau} W. de Launey: 
Circulant $\text{GH}(p^2, \Z_p)$ Exist for All Primes $p$. \textit{Graphs Comb.} \textbf{8} (1992), 317 -- 321.

\bibitem{leu1} K. H. Leung, S. L. Ma: 
Constructions of partial difference sets and relative difference sets on $p$-groups.
\textit{Bull. London Math. Soc.} \textbf{22} (1990), 533--539.

\bibitem{leu2} K. H. Leung, S. Ling, S. L. Ma:
Constructions of semi-regular relative difference sets.
\textit{Finite Fields Appl.} \textbf{7} (2001), 397--414.

\bibitem{lev} W. LeVeque: 
\textit{Topics in Number Theory}, Vol. 1, Addison-Wesley, Reading, Mass., 1956.

\bibitem{schmidt1} B. Schmidt:
A Survey of Group Invariant Butson Matrices and Their Relation to Generalized Bent Functions and Various Other Objects. \textit{Radon Series on Computational and Applied Mathematics} \textbf{23} (2019), 241 -- 251.

\end{thebibliography}
\end{document}